\documentclass[12pt]{article}

\usepackage{epsfig,amsfonts,amsthm,amssymb,latexsym,amsmath}

\theoremstyle{plain}
\newtheorem{theorem}{Theorem}[section]
\newtheorem{corollary}[theorem]{Corollary}
\newtheorem{lemma}[theorem]{Lemma}
\newtheorem{maintheorem}{Theorem}
\theoremstyle{remark}
\newtheorem{remark}[theorem]{Remark}

\theoremstyle{definition}
\newtheorem{definition}[theorem]{Definition}

\newcommand{\field}[1]{\mathbb{#1}}

\renewcommand{\natural}{\field{N}}

\newcommand{\be} {\beta}        
\newcommand{\ga} {\gamma}    \newcommand{\Ga}{\Gamma}
\newcommand{\de} {\delta}       

\newcommand{\eps}{\varepsilon}

\newcommand{\la} {\lambda}      \newcommand{\La}{\Lambda}

\newcommand{\si} {\sigma}

       \newcommand{\Om}{\Omega}

\newcommand{\N}{\mathbb{N}}
\newcommand{\R}{\mathbb{R}}
\newcommand{\supp}{\operatorname{supp}}
\newcommand{\diam}{\operatorname{Diam}}

\newcommand{\lip}{\operatorname{Lip}_1(X;\R)}

\newcommand{\SM}{{\mathcal M}}

\newcommand{\cK}{\mathcal{K}}

\newcommand{\cP}{\mathcal{P}}

\newcommand{\cB}{\mathcal{B}}

\newcommand{\cF}{\mathcal{F}}

\newcommand{\cA}{\mathcal{A}}

\newcommand{\p}{\field{P}}
\newcommand{\SK}{{\mathcal K}}
\newcommand{\SF}{{\mathcal F}}

\begin{document}

\title{On weakly hyperbolic iterated function systems}

\author{Alexander\ Arbieto\footnote{Corresponding author.}\thanks{A.A. was supported by CNPq, FAPERJ, CAPES and PRONEX-DS.}
\ and\ Andr\'e\ Junqueira\thanks{A.J. was supported by CNPq and FAPERJ.}
\ and\ Bruno\ Santiago\thanks{B.S. was supported by CNPq and FAPERJ.} 
}
\date{}

\maketitle

\baselineskip=0.9\normalbaselineskip

\begin{abstract}
We study weakly hyperbolic iterated function systems on compact {metric} spaces, as defined by Edalat in \cite{E}, 
but in the more general setting of compact parameter space. We prove the existence of attractors, both in the topological and measure theoretical viewpoint 
and the ergodicity of invariant measure. We also define weakly hyperbolic iterated function systems for complete {metric} 
spaces and compact parameter space, extending the above mentioned definition. Furthermore, we study the question of existence of attractors in this setting. 
Finally, we prove a version of the results in \cite{BV1}, about drawing the attractor (the so-called the chaos game), for compact parameter space.
\end{abstract}

\smallskip
\noindent
{\small{\bf Keywords:}  
Iterated function systems, attractors, chaos game.}

\smallskip
\noindent
{\small{\bf Mathematical subject classification:} 
Primary: 37B99; Secondary: 37A05.}

\baselineskip=\normalbaselineskip

\section{Introduction}

Iterated function systems (IFS) were introduced in \cite{H} (although some results appeared earlier in \cite{Wi}), as a unified way of generate a broad class of fractals. 
Nowadays, such systems occurs in many places in mathematics and other scientific areas, like image processing \cite{B2}. IFS can also be considered as skew-products over the shift map. 
Therefore, they can also be considered as random dynamical systems, like in \cite{CSS}. 

In \cite{H}, Hutchinson introduced the theory of hyperbolic IFS. He considered a finite collection of contractions over a complete metric space. He was interested in constructing attractors, both in the topological and measure-theoretical viewpoint. 
Thus, he built some operators from the IFS, which nowadays are called the Hutchinson-Barnsley and transfer operators.
This theory and the fractal theory was largely disseminated by the book \cite{B1}.

After this seminal work of Hutchinson, many authors proposed several generalizations of his results. One direction was to weaken the hyperbolicity assumption, allowing some weak forms of contraction. For instance, we have the so-called average 
contraction with respect to a probability measure, studied in \cite{BDEG} and \cite{CSS}. Also, we have the $\phi$-contractions studied by \cite{JL} and \cite{M}.

Following this line of research Edalat \cite{E} defined the notion of weakly hyperbolic IFS (see the  definition \ref{wh}) as a finite collection of maps on a compact metric space such that the diameter of the space by any combination of the maps
goes to zero. This definition allow some non-contractions which were ruled out in the previous settings to obtain a topological attractor. 

Another way to extend the results of Hutchinson is to enlarge the parameter space. In Hutchinson's paper the parameter space is always finite. 
In \cite{Ha}, this theory was extended to the case when the parameter space is an infinite countable set. In \cite{L} and \cite{Me} the authors consider compact metric spaces as the parameter {space}. 
However, in those contexts, only uniform contractions and average contractions {were} studied.

One of the purposes of this article is to study these questions in the setting of weakly hyperbolic IFS with compact parameter space, thus unifying and extending some of the previous results. 
In particular, we obtain the existence of topological and measure-theoretical attractors. Moreover, we extend the notion of weakly hyperbolic IFS for complete metric spaces and we discuss and give partial results about the existence of such attractors.

Let us make some comments about our proofs. In the compact case, the idea is to show that our definition satisfies a well known property called point fibered property as mentioned in \cite{BV2} by Barnsley and Vince. This property, in a stronger form, was also studied by Mat\'e, in \cite{M}, with the name of property $(P*)$. So, one step is to prove that weak hyperbolicity implies this property. We stress that in the complete case, we still obtain the existence of topological attractors using weak hyperbolicity. However, we still cannot prove the existence of measure-theoretical attractor using only weak hyperbolicity. Nevertheless, we also have some partial results about this.

Moreover, in the compact case we prove ergodicity of the measure-theoretical attractor for the systems involved. Also, inspired by the work of Barnsley-Vince in \cite{BV1}, we also prove that most orbits can draw the attractor (see the precise definition below) with respect to some special measures in the parameter space. We remark that this property is called ``chaos game'' in Barnsley-Vince's work. 

The rest of this introduction is devoted to give precise definitions and statements of our results.

\subsection{Definitions}

{Let $\La$ be a compact metric space and $X$ be a complete metric space}. 
{A continuous map $w:\Lambda \times X \to X$ is called an \emph{Iterated Function System} (IFS for short)}. 
The space $\Lambda$ is called the \emph{parameter space} and $X$ is called the \emph{phase space}. 
The space $\La^{\N}$ of infinite words with alphabet in $\La$, endowed with the product topology will be denoted by $\Om:=\La^{\N}$. Given a fixed parameter $\la\in\La$, we will denote by $w_{\lambda}:X \to X$ the partial map generated by this parameter, 
which is defined by $w_{\lambda}(x):=w(\lambda,x)$. 

In this paper we shall investigate IFS with compact parameter space.

Let us denote the map $w_{\lambda_{1}...\lambda_{n}}:=w_{\lambda_{1}}\circ ... \circ w_{\lambda_{n}}$, where $(\lambda_{1},...,\lambda_{n})\in\La^{n}$. For each $n\in \N$ we denote by $w^{n}$ the IFS 
from $\Lambda^{n}\times X$ to $X$, given by $$w^n(\lambda_{1},...,\lambda_{n},x):=w_{\lambda_{1},...,\lambda_{n}}(x).$$ 

{Let us recall the definition of weakly hyperbolic IFS, as introduced by Edalat in \cite{E}.}

\begin{definition}
If $X$ is a compact metric space and $\Lambda$ is any compact metric space then we say that an IFS $w:\Lambda\times X\to X$ is \emph{weakly hyperbolic} if for every $\sigma\in \Om$ we have:
$$\lim_{n\to \infty}Diam(w_{\sigma_{1}...\sigma_{n}}(X))=0$$
\label{wh}
\end{definition}

{In \cite{E} Edalat considered weakly hyperbolic IFS with finite parameter space. One of the goals of this paper is to extend his results to the more general
setting of arbitrary compact metric spaces as a parameter space}

\subsection{The Topological Attractor}
First, we recall the Hausdorff topology. Let us denote by $\mathcal{K}(X)$ the family of all compact subsets of $X$. We endow it with the Hausdorff metric as follows. Let $d(x,F)=\inf \{d(x,y);y\in F\}$. The Hausdorff  metric is given by
$$d_H(A,B)=\sup \{d(a,B),d(b,A): a\in A, b\in B \}\,\,\textrm{for}\,\,A,B \in \mathcal{K}(X)$$
If $X$ is a complete (resp. compact) metric space, it can be proved (see \cite{B1}) that $(\mathcal{K}(X),d_H)$ is also a complete (resp. compact) metric space. 
The \emph{Hutchinson-Barnsley operator} $\mathcal{F}: \mathcal{K}(X)\longrightarrow \mathcal{K}(X)$ is given by:
$$\mathcal{F}(A):= \bigcup_{\lambda \in \Lambda}w_{\lambda}(A)=w(\La\times A),\,\,\textrm{for}\,\,A\in \mathcal{K}(X).$$

\begin{definition}
An IFS $w$ has an \emph{attractor} $A\in \mathcal{K}(X)$, if there exists an open neighborhood $U$ of $A$ (the \emph{basin of attraction})
such that $\mathcal{F}^{n}(B)\to A$ in the Hausdorff topology for every $B\in \mathcal{K}(X)$, with $B\subset U$. If $A\in \mathcal{K}(X)$ is a fixed point of $\mathcal{F}$ then we say that $A$ is an \emph{invariant set} by $w$. If $U=X$ then the IFS has a \emph{global attractor}.
\end{definition}

{We shall deal with attractors which might not be global attractors only in Section~\ref{s.caos}.} Thus, to simplify the notation, we make the following convention: 
when we say that an IFS has an attractor, but we do not make any comment about the basin, we shall be talking about global attractors. 

Our first result gives the existence of global attractors for weakly hyperbolic IFS.

\begin{maintheorem}
\label{thm:2}
Let $w$ be a weakly hyperbolic IFS on a compact metric space $X$ and with a compact parameter space $\Lambda$. Then $\mathcal{F}$ has an attractor $K$ that is also a compact invariant set. 
Furthermore, we have that $w_{\sigma_{1}}\circ ... \circ w_{\sigma_{n}}$ has a unique fixed point $\textrm{for all}\,\, \sigma \in \Om\,\,\textrm{and}\,\, n\geq1$ and also $K$ is the closure of these fixed points.
\end{maintheorem}

\subsection{The Measure-Theoretical Attractor}
First, we recall the topologies on the measure space.
Let $(X,d)$ be a complete and separable metric space and consider the space 
$$\lip=\{f:X\to \R: |f(x)-f(y)|\leq d(x,y)\,\,\textrm{for all}\,\,x,y\in X\}.$$
Let $\mathcal{M}(X)$ be the set of the Borel probability measures $\mu$ such that $\mu(f):=\int_{X}f d\mu < +\infty$ for each $f\in \lip$. 
We define the \emph{Hutchinson metric} in $\mathcal{M}(X)$ by:
$$H(\nu,\mu)=\sup\left\{\left |\int_{X}fd\nu - \int_{X}fd\mu \right|;f\in\lip\right\}.$$
In \cite{K}, Kravchenco characterized the completeness of $\mathcal{M}(X)$ with the Hutchinson metric:

\begin{theorem}
Let $X$ be a separable metric space.
Then the space $(\mathcal{M}(X),H)$ is complete if and only if $X$ is complete.
\label{completitude}
\end{theorem}
\begin{remark}
We remark that Hutchinson used a different measure space in his paper and proved his result on the existence of a measure-theoretical attractor for a contractive IFS using Banach's fixed point theorem. Nevertheless, in \cite{K}, 
Kravchenco proved that the space considered by Hutchinson {is} not complete. Kravchenco defined the space $\SM(X)$ as above and proved the completeness of this space. Then, he proved that Hutchinson's arguments work with this space. 
\end{remark}

\begin{remark}
The above result also works in complete but non separable spaces, provided that we restrict ourselves to measures with separable support. See \cite{K} for more details. 

\end{remark}

Now, let us recall the weak$^{*}$ topology:

Let us denote by $C_{b}(X)$ the set of bounded and continuous functions $f:X\to \R$.
Given $\eps>0$, $\nu \in \mathcal{M}(X)$ and $f_{1},...,f_{k}\in C_{b}(X)$ we define: 
$$V(\nu,\eps,k):=\{\mu \in \mathcal{M}(X):|\mu(f_{j})-\nu(f_{j})|<\eps, j=1,...,k\}.$$
The \emph{weak$^*$ topology} is the topology generated by the basis $V(\nu,\eps,k)$ for each $\eps,k,\nu$. Furthermore, we have that $\mu_{n}$ converges for $\mu$ in the weak$^{*}$ topology if and only if
$\mu_{k}(f)\longrightarrow \mu(f)$ for every $f\in C_{b}(X)$. The relation between the weak$^*$ topology and the Hutchinson topology is given by next theorem. A proof can be found in \cite{K}.

\begin{theorem}
The Hutchinson topology and the weak$^{*}$ topology are equivalent if and only if $Diam(X)<+\infty$. Furthermore, if $Diam(X)=\infty$ then the Hutchinson topology is finer than weak$^*$ topology.
\label{fracaversushutchinson}
\end{theorem}

When $(X,d)$ is a compact metric space, we have the following result on the metrizability of $\mathcal{M}(X).$ The proof can be found in \cite{W}.

\begin{theorem}
If $X$ is a compact metric space and $\{f_{n}\}_{n\in \N}$ is a dense set in the unit sphere of $C(X)$ with the uniform metric, then the function:
$$D(\nu,\mu):=\sum_{n=1}^{\infty}\frac{1}{2^{n}}|\int_{X} f_{n}d\nu - \int_{X} f_{n}d\mu|$$
is a metric in $\mathcal{M}(X)$ generating the weak$^*$ topology.
\label{metricacompacta}
\end{theorem}

Under the measure-theoretical point of view we also have a notion of attractor. To explain this notion, we shall define the \emph{transfer operator}:

\begin{definition}
Let $p$ be a probability {measure} in $\Lambda$. We define the \emph{Transfer Operator} $T_p:\SM(X)\to\SM(X)$ by the formula: 
$$T_p(\mu)(B):=\int_{\Lambda}\mu(w_{\lambda}^{-1}(B))dp(\lambda),$$
for every Borel set $B$ and for each measure $\mu \in \mathcal{M}(X)$. If a measure $\mu \in \mathcal{M}(X)$ is a fixed point of the transfer operator we say that $\mu$ is an \emph{invariant measure} for $w$.
\end{definition}

\begin{remark}
Sometimes we will omit the set $B$ in the definition and write:
$$T_p(\mu):=\int_{\Lambda}w_{\lambda}^{*}(\mu)dp(\lambda).$$
\end{remark}
where $*$ is the push-forward operator.

\begin{definition}
We say that a probability  $\nu \in \mathcal{M}(X)$ is a \emph{measure-theoretical attractor} for $w$ if $T_p^{n}(\mu)\stackrel{n\to\infty}{\longrightarrow}\nu$ in the Hutchinson metric for all $\mu\in \mathcal{M}(X)$.
\end{definition}

Our result giving the existence of a unique measure-theoretical attractor in the is the following.

\begin{maintheorem}
\label{teoremaatratorglobalmetrico}
If $X$ is a compact metric space and $w$ is a weakly hyperbolic IFS, then $w$ has a measure-theoretical attractor $\nu\in \mathcal{M}(X)$ which is the unique fixed point of the transfer operator.
Furthermore, $p(U)>0$ for every open set $U\subset \Lambda$ then we have that $supp(\nu)=K$, where $K$ is the topological attractor given by Theorem~\ref{thm:2}.
\end{maintheorem}

If $\nu$ is an invariant measure for an IFS $w$, then we can define the ergodicity of $\nu$. This notion is related with the ergodic theorem for an IFS. See \cite{CSS} for details.

\begin{definition}
Fix $p\in \mathcal{M}(\Lambda)$ and $\mathbb{P}=p^{{\natural}}$ {the product measure}. We say that an invariant measure ${\mu}$ for $w$ is ergodic if for every continuous function $f:X\to \R$, every $x\in X$ and $\p$-almost every $\si\in\Om$ we have:
$$\lim_{n\to\infty}\frac{1}{n}\sum_{j=0}^nf(w_{\sigma_j}\circ\dots\circ w_{\sigma_1}(x))=\int_X fd\mu.\footnote{We shall use the convention that $w_{\sigma_j}\circ\dots\circ w_{\sigma_1}(x)=x$, if $j=0$.}$$
\end{definition}

Our next result is about the ergodicity of the measure-theoretical attractor.

\begin{maintheorem}
\label{ergodicidadedamedida}
If $X$ is a compact metric space and $w$ is a weakly hyperbolic IFS, then its unique measure-theoretical attractor is ergodic.
\end{maintheorem}

\subsection{The complete case}

We propose the following definition as an extension of the concept of weakly hyperbolic IFS.

\begin{definition}
\label{$w^{*}$h}
Let $w:\Lambda\times X\to X$ be a continuous IFS, where $(X,d)$ is a complete metric space. We say that $w$ is \emph{weakly$^*$ hyperbolic}  if for all $x,y\in X$ and $\sigma\in \Om$ we have:
$$\lim_{n\to +\infty}d(w_{\sigma_{1}...\sigma_{n}}(x),w_{\sigma_{1}...\sigma_{n}}(y))=,0$$
where the convergence is assumed to be uniform in $\Om$ and locally uniform in $X$. In other words, there exists $\eta>0$ such that for all $\eps>0$ there exists $n_{0}=n_{0}(\eps)$ such that if $n\geq n_{0}$ then
$$d(w_{\sigma_{1}...\sigma_{n}}(x),w_{\sigma_{1}...\sigma_{n}}(y))<\eps,$$ for all $\sigma\in\Om$ and $x,y$ such that $d(x,y)<\eta.$
\end{definition}
In Section~\ref{completo} {we will prove} that if $X$ is compact, then an IFS $w$ is weakly$^*$ hyperbolic if and only if $w$ is weakly hyperbolic. 
We state here results in the complete case. 

Our result concerning the existence of a topological global attractor in the complete case is the following.

\begin{maintheorem}
\label{thm:1}
Let $w$ be a weakly hyperbolic IFS on a complete metric space $X$ and with a compact parameter space $\Lambda$. Assume that $(\mathcal{K}(X),d_H)$ is $\eps$-chainable for every $\eps>0$. 
Then $\mathcal{F}$ has an attractor $K$ that is also a compact invariant set.
\end{maintheorem}

For the definition of an $\eps$-chainable metric space, we refer the reader to Section~\ref{completo}. However, we remark that this theorem can be applied when $X$ is a Banach space or a complete Riemannian manifold.

Regarding the existence of attractors from the measure-theoretical viewpoint, we have the following result:

\begin{maintheorem}
\label{medidainvariantecompleto}
Let $(X,d)$ be a complete metric space, uniformly $\eps$-chainable on balls and with $(\SK(X),d_H)$ $\eps$-chainable, for every $\eps>0$. If $w$ is a weakly hyperbolic IFS, then there exists a unique invariant measure $\nu\in \mathcal{M}(X)$ 
such that $supp(\nu)\subset K$ and in fact we get that $supp(\nu)=K$ if $p(U)>0$ for each $U\subset \Lambda$ open, where $K$ is the attractor given by Theorem \ref{thm:1}. Furthermore, if $\mu\in\mathcal{M}(X)$ has compact support then $T_p^{n}(\mu)\stackrel{n}{\longrightarrow}\nu$ in the Hutchinson metric.
\end{maintheorem}
 
\subsection{Drawing the attractor}

An \emph{orbit of the IFS} starting at some point $x$ is a sequence $\{x_k\}_{k=0}^{\infty}$ such that $x_0=x$, $x_{k+1}=w_{\si_k}(x_k)$, for some $\si=\{\si_k\}_{k=1}^{\infty}\in\Om$. If an IFS $w:\Lambda\times X\to X$ has an attractor $A$, we say that an orbit starting at $x$ \emph{draws} the attractor if the tails {of this orbit} are getting close, in the Hausdorff metric, to the attractor, i.e. if $$A=\lim_{k\to\infty}\{x_n\}_{n=k}^{\infty},\:\:\:\textrm{in the Hausdorff metric}.$$ 

{This concept is inspired by the so-called chaos game, studied in \cite{BV1} in the case of finite parameter space. Our last result says something about orbits of the IFS that draws the attractor. As in \cite{BV1}, it is not necessary to make any assumption of hyperbolicity
for the IFS, only the existence of a local attractor suffices. Nevertheless, in order to be able to prove a result for the case of arbitrary compact parameter space we needed to consider probability measures in the parameter space that possesses a 
uniform lower bound for the size of balls. We called such measures \emph{fair}. See Section~\ref{s.caos} for details.} 

\begin{maintheorem}
\label{caosgamecor}
Let $X$ be a proper complete metric space, $\Lambda$ be a compact metric space.
Consider $p\in\SM(\La)$ a fair probability measure, and $\p:=p^{\N}\in\SM(\Om)$.
Assume that $w:\Lambda\times X\to X$ is a continuous IFS. Then given $x\in X$, a $\p$-total probability set of orbits of $x$ draws the attractor $K$ of $w$.
\end{maintheorem}

\section{Proof of Theorem \ref{thm:2}}\label{topattractorcompacto}

{In \cite{E} Edalat proved the existence attractors for weakly hyperbolic IFS in the context of finite parameter space. His argumentes use concepts of
graph theory and are not available in our setting. Our strategy, instead, is to take advantage of the compactness of the phase space to show that $\diam(w_{\si_1...\si_n}(X))$ goes to zero \emph{uniformly} in $\Om$. We then use this fact combined with the more axiomatic approach of Mat\'e in \cite{M}
to show the existence of the attractor. To prove that this attractor is the closure of the fixed points of the partial maps $w_{\sigma_1...\sigma_n}$ we apply
a fixed point theorem of Jachymski \cite{J}.}

\begin{lemma}
For each $n\in \N$, the function $(\lambda_{1},...,\lambda_{n})\in\Lambda^{n}\mapsto\diam(w_{\lambda_{1}...\lambda_{n}}(X))\in\R$ is uniformly continuous with respect to the maximum metric.
\label{continuidadeproduto}
\end{lemma}
\begin{proof}
{Fix $\eps>0$}. Let us denote by $\rho$ the metric of $\Lambda$ and $d$ the metric of $X$. Let us define for $A\subset X$ and $t>0$:
$$B(A,t):=\{y\in X:d(y,A)\leq t\}$$
Since $w^{n}$ is uniformly continuous, there exists $\delta>0$ such that if $$\max\{\rho(\lambda_{1},\lambda_{1}^{*}),...,\rho(\lambda_{n},\lambda_{n}^{*}),d(x,y)\}<\delta$$
then 
$$d((w_{\lambda_{1}...\lambda_{n}})(x),(w_{\lambda_{1}^{*}...\lambda_{n}^{*}})(y))<\eps.$$

Take $(\lambda_{1},...,\lambda_{n})$ and $(\lambda_{1}^{*},...,\lambda_{n}^{*})$ in $\Lambda^{n}$ such that  $$\max\{\rho(\lambda_{1},\lambda_{1}^{*}),...,\rho(\lambda_{n},\lambda_{n}^{*}))\}<\delta.$$ We claim that:
\begin{enumerate}
\item $w_{\lambda_{1}...\lambda_{n}}(X)\subset B\left(w_{\lambda_{1}^{*}...\lambda_{n}^{*}}(X),\eps\right).$
\item $w_{\lambda_{1}^{*}...\lambda_{n}^{*}}(X)\subset B\left(w_{\lambda_{1}...\lambda_{n}}(X),\eps\right).$
\end{enumerate}

Indeed, if $y\in w_{\lambda_{1}...\lambda_{n}}(X)$ then we can write $y=w_{\lambda_{1}...\lambda_{n}}(x)$ where $x\in X$. Hence, if we define $y^{*}:=w_{\lambda_{1}^{*}...\lambda_{n}^{*}}(x)$, we have
$$d(y,y^{*})=d(w_{\lambda_{1}...\lambda_{n}}(x),w_{\lambda_{1}^{*}...\lambda_{n}^{*}}(x))<\eps.$$
This shows that $y\in B\left(w_{\lambda_{1}^{*}...\lambda_{n}^{*}}(X),\eps\right).$ The proof of (2) is similar. This finishes the lemma.
\end{proof} 

{The following is the key lemma of this section. We shall prove that $\diam(w_{\sigma_{1}...\sigma_{n}}(X))$ goes to zero 
uniformly with respect $\sigma\in\Omega$.}

\begin{lemma}\label{lemaprincipal}
Let $w$ be an IFS on a compact metric space $X$ with a compact parameter space. Then the following are  equivalent.
\begin{enumerate}
\item $w$ is weakly hyperbolic 
\item Given $\eps >0,$ there exists $n_{0}=n_{0}(\eps) \in\N$ such that for all 
$n\geq n_{0}$ and $\sigma \in \Om$ we have
$$\diam(w_{\sigma_{1}...\sigma_{n}}(X))<\eps$$
\end{enumerate}
\end{lemma}

\begin{proof}
If $w$ satisfies (2), then it is obvious that $w$ satisfies (1). So, it is enough to prove that (1) implies (2). Let us suppose that (2) is false. Then there exists $\eps_{0}>0$, a sequence $(n_{k})\longrightarrow +\infty$ and a sequence of words (with alphabet in $\Lambda$):\\
$$(\sigma_{1}^{1},\sigma_{2}^{1},...,\sigma_{n_{1}}^{1}),(\sigma_{1}^{2},\sigma_{2}^{2},...,\sigma_{n_{2}}^{2}),...$$ such that:

\begin{equation}
\label{hip}
\diam(w_{\sigma_{1}^{k}...\sigma_{n_{k}}^{k}}(X))\geq \eps_{0}\,\,\textrm{for any}\,\,k\in \N.
\end{equation}
Thus we have the following matrix builded with these words:
\begin{eqnarray*}
\sigma_{1}^{1} \sigma_{2}^{1} &...& \sigma_{n_{1}}^{1}\\
\sigma_{1}^{2} \sigma_{2}^{2} &...& \sigma_{n_{1}}^{2} ... \sigma_{n_{2}}^{2}\\
\vdots& & \\
\sigma_{1}^{k} \sigma_{2}^{k} &...& \sigma_{n_{1}}^{k} ... \sigma_{n_{2}}^{k} ... \sigma_{n_{k}}^{k}\\
\vdots& &
\end{eqnarray*}

Now, using the compactness of $\Lambda$ and a diagonal argument we can obtain that each column of the matrix is convergent in $\Lambda$.
Indeed, the first column is a sequence in $\Lambda$ and then there exists a set $\N_{1}\subset \N$ such that $\{\sigma_{1}^{k}\}_{k\in \N_{1}}$ is convergent in $\Lambda.$ Analogously, there exists a set $\N_{2}\subset \N_{1}\subset \N$ such that $\{\sigma_{2}^{k}\}_{k\in \N_{2}}$ is convergent in $\Lambda$ and so on. In this way, we obtain a nested sequence of sets
$$\N\supset \N_{1}\supset \N_{2}\supset ...$$
and if we define a set $\N^{*}$ such that its first element is the first element of $\N_{1}$, its second element is the second element of $\N_{2}$ and so on, we obtain that the matrix  $\{\sigma_{j}^{k}\}_{k\in \N^{*},j\leq n_{k}}$ has all columns
convergent in $\Lambda$. Therefore, for simplicity, we can suppose that the initial matrix has all columns convergent and we can define  $\sigma=(\sigma_{1},\sigma_{2},...)\in \Om$ where each element of this sequence is the limit of the associated column. 
So, to finish the proof it is enough to prove that this sequence does not satisfy the definition of weak hyperbolicity. Indeed, fix $m\in \N$ and consider the word $(\sigma_{1},...,\sigma_{m}).$ Using that $(n_{k})\to \infty$ we have $m<n_{k}$, 
for every $k$ sufficiently large. Then it follows from (\ref{hip}) that
$$\diam(w_{\sigma_{1}^{k}...\sigma_{m}^{k}}(X))\geq \eps_{0},$$
for $k$ sufficiently large.
Since $(\sigma_{1}^{k},...,\sigma_{m}^{k})\stackrel{k}{\longrightarrow}(\sigma_{1},...,\sigma_{m})$ in the maximum metric, it follows from Lemma \ref{continuidadeproduto} that  $\diam(w_{\sigma_{1}...\sigma_{m}}(X))\geq \eps_{0}$. Since $m$ is arbitrary, this contradicts the definition of weak hyperbolicity and completes the proof.
\end{proof}

{Let us recall a property defined by Mat\'e in \cite{M}, in his axiomatic approach to the existence of attractors for IFS.}

\begin{definition}
Let $w:\La\times X\to X$ be an IFS. For each $\si\in\Om$, $n\in\N$, and $x\in X$, define $\Gamma(\sigma,n,x):= w_{\sigma_{1}...\sigma_{n}}(x)$. We say that $w$ satisfies \emph{Property P}$^*$ if
\begin{eqnarray}
\Gamma(\sigma):= \lim_{n \to \infty}\Gamma(\sigma,n,x)
\end{eqnarray}
exists for every $\sigma \in \Om$ and $x \in X$, does not depend on $x$ and is uniform on $\sigma$ and $x\in X$.

\begin{remark}
In \cite{BV1} (for instance) there is the notion of point fibered IFS, which is a weaker version of property P$^*$, since it do not require the limit to be uniform on $\si\in\Om$ and $x\in X$.
\end{remark}  
 
\end{definition}
\begin{corollary}

\label{pestrela}
Every weakly hyperbolic IFS $w:\La\times X\to X$, with $X$ and $\La$ compact metric spaces, satisfies property P$^*$.
\end{corollary}

\begin{proof}
Take $x\in X$ and $\eps>0$. By Lemma \ref{lemaprincipal} we have that there exists $n_{0}=n_{0}(\eps)$ such that:
$$\diam(w_{\sigma_{1}...\sigma_{n}}(X))<\eps,$$ 
for every $\si\in\Om$ and every $n\geq n_0$. Observe that $$\Gamma(\sigma,n,x)\in w_{\sigma_{1}...\sigma_{n}}(X)$$ and $$\Gamma(\sigma,n+p,x)\in w_{\sigma_{1}...\sigma_{n+p}}(X)\subset w_{\sigma_{1}...\sigma_{n}}(X),$$ 
and therefore we have that $d(\Gamma(\sigma,n+p,x),\Gamma(\sigma,n,x))<\eps\,\,\textrm{for all}\,\,n\geq n_{0}\,\,\textrm{and}\,\,p\in \N.$ 
Then, the sequence $\Gamma(\sigma,n,x)$ is Cauchy and thus convergent for all  $x\in X$ and $\sigma \in \Om$. {Since} $n_{0}$ does not depend on $\sigma$ we obtain the uniformity on $\sigma$. Now, take $\sigma \in \Om$ and $x,y \in X$. Then we have:
$$\Gamma(\sigma,n,x),\Gamma(\sigma,n,y) \in  w_{\sigma_{1}}\circ ... \circ w_{\sigma_{n}}(X),$$
and so
$$\lim_{n \to \infty}d(\Gamma(\sigma,n,x),\Gamma(\sigma,n,y))=0,$$
which shows that the limit does not depend on $x$. This completes the proof.
\end{proof}

Property P$^{*}$, in the case $\La=\{1,...,N\}$, was proved by \cite{M} to be a sufficient condition for the existence of an attractor. Here, we will prove this in the more general case of $\La$ being an arbitrary compact space. 
{For adapting his arguments we need some preparatory lemmas. The first one proves that the Hutchinson-Barnsley operator is continuous. The proof we give here also works in the case where $X$ is complete but not necessarily compact 
and will be used later in this paper.}

\begin{lemma}
{Let $X$ be a complete metric space and $\Lambda$ a compact metric space}. If $w:\Lambda\times X\to X$ is continuous, then
{the associated Hutchinson-Barnsley operator} $\SF$ is also continuous.
\label{continuidade}
\end{lemma}
\begin{proof}
Fix a compact set $K\subset X$. Take an $\eps>0$. Since $w$  is continuous and $\La$ is compact, there exists $\beta>0$ such that if $x\in K$ and $y\in X$ with $d(x,y)<\,\beta$, then $$d(w_{\lambda}(x),w_{\lambda}(y))<\eps,\:\:\textrm{for every}\:\:\lambda\in\La.$$

Assume that $A\in\cK(X)$ is such that $d_H(A,K)<\beta$. Let $x$ be a point in $K$ and take $a\in A$ with $$d(a,x)=d(x,A)<\beta.$$ Then $$d(w_{\la}(x),w_{\la}(A))\leq d(w_{\la}(x),w_{\la}(a))<\eps,\:\:\textrm{for every}\:\lambda\in\La.$$
In a similar manner we show that for every $a\in A$, $$d(w_{\la}(a),w_{\la}(K))<\eps,\:\:\textrm{for every}\:\la\in\La.$$ This proves that $$d_H(w_{\la}(A),w_{\la}(K))\leq\eps,\:\:\textrm{for every}\:\la\in\La,$$ and thus $$d_H(\SF(A),\SF(K))\leq\eps,$$ 
which establishes the result.
\end{proof}  

Observe that Corollary \ref{pestrela} defines a function $\Ga:\Om\to X$, given by $$\Ga(\si)=\lim_{n\to\infty}\Ga(\si,n,x),\:\:\textrm{for any}\:\:x\in X.$$
As in \cite{M} we have the following. 
\begin{lemma}
The map $\Gamma:\Om\to X$ is continuous in the product topology on $\Om$.
\label{continuidadeGamma}
\end{lemma}
\begin{proof}
Let us denote by $\rho$ the metric of $\Lambda$. Fix $\sigma\in \Om$ and $\eps>0$. By Corollary \ref{pestrela} we have that there exists $m=m(\eps)\in \N$ such that
$$d(w_{\sigma_{1}}\circ ... \circ w_{\sigma_{m}}(x),\Gamma(\sigma))<\eps\,\,\,\textrm{for all}\,\, \sigma\,\,\textrm{and}\,\,x.$$
Now, using that $w^{m}$ is continuous we get $a>0$ such that if 
$$\rho(\sigma_{1}^{*},\sigma_{1})<a,...,\rho(\sigma_{m}^{*},\sigma_{m})<a,$$ then
$$d(w_{\sigma_{1}...\sigma_{m}}(x),w_{\sigma_{1}^{*}...\sigma_{m}^{*}}(x))<\eps\,\,\textrm{for all}\,\,x.$$
Let $U$ be the neighborhood of $\sigma$ in the product topology given by:
$$U=B_{\rho}(\sigma_{1},a)\times ... \times B_{\rho}(\sigma_{m},a)\times \Lambda \times...$$
Therefore, if $\sigma^{*}\in U$, then:
\begin{eqnarray*}
&&d(\Gamma(\sigma^{*}),\Gamma(\sigma))\\
&\leq&d(\Gamma(\sigma),w_{\sigma_{1}...\sigma_{m}}(x))\nonumber \\
&+&d(w_{\sigma_{1}...\sigma_{m}}(x),w_{\sigma_{1}^{*}...\sigma_{m}^{*}}(x))+d(\Gamma(\sigma^{*}),w_{\sigma_{1}^{*}...\sigma_{m}^{*}}(x))\nonumber \\
&<&3\eps. \nonumber \\
\end{eqnarray*}
This shows that $\Gamma$ is continuous.
\end{proof}

Finally, we shall use the fixed point theorem of Jachymski \cite{J}. This theorem is a generalization of Banach's fixed point theorem.
Before state it we need a definition.
\begin{definition}
\label{contracaoassintotica}
Let $(X,d)$ be a metric space. We say that a map $T:X\to X$ is an \emph{asymptotic contraction} if $d(T^{n}(x),T^{n}(y))\stackrel{n\to +\infty}{\longrightarrow}0$ for all $x,y\in X$ and there exists $\eta>0$ such that this convergence is uniform if 
$d(x,y)\leq \eta$.
\end{definition}

\begin{theorem}[Jachymski]
\label{Jachymski}
Suppose that $(X,d)$ is a complete metric space and $T:X\to X$ is a continuous asymptotic contraction. Then there exists $x\in X$ such that:
$$d(T^{n}(y),x)\stackrel{n\to +\infty}{\longrightarrow}0\,\,\,\textrm{for all}\,\,\,y\in X.$$
\end{theorem}

For a proof, se \cite{J}.\\

\begin{proof}[Proof of Theorem \ref{thm:2}] 
{Notice that} if $A\in \mathcal{K}(X)$ then we can write
$$\mathcal{F}^{n}(A)=\bigcup_{\sigma \in \Om} w_{\sigma_{1}...\sigma_{n}}(A).$$
Define 
$$K:=\Gamma(\Om)=\{\lim_{n \to \infty}\Gamma(\sigma,n,x):\sigma \in \Om\}$$
and notice that, by Lemma \ref{continuidadeGamma}, $K$ is a compact set. So, it {remains} to prove that $K$ is an attractor. In fact, given $B\subset X$ a compact set and $\eps >0$ we have by Corollary~\ref{pestrela} that there exists $n_{0}=n_{0}(\eps)$ such that:
$$d(\Gamma(\sigma,n,x),\Gamma(\sigma))<\eps\,\,\,\textrm{for all}\,\,n\geq n_{0},\,\, \sigma\in\Om,\,\,\textrm{and}\,\,x\in B.$$
Fix $n\geq n_{0}$. Then, for all $y\in \mathcal{F}^{n}(B)$ there exists $z\in K$ such that $d(y,z)<\eps$ and analogously given $z\in K$ there exists $y\in \mathcal{F}^{n}(B)$  such that $d(y,z)<\eps$. 
This shows that $d_H(\mathcal{F}^{n}(B),K)<\eps$ if $n\geq n_{0}$. Therefore, $\lim_{n\to +\infty}d_H(\mathcal{F}^{n}(B),K)=0$.
Using that $\mathcal{F}$ is continuous we have that $K$ is the unique compact invariant set of $w$.

To prove the statement on the fixed points, take $g=w_{\sigma_{1}}\circ ... \circ w_{\sigma_{n}}$ with $\sigma \in \Om$ and $n\geq 1$. Then we have that:
$$g^{m}(x)=w_{\sigma_{1}}\circ ... \circ w_{\sigma_{n}}\circ...\circ w_{\sigma_{1}}\circ ... \circ w_{\sigma_{n}}(x)$$ where the first block appears $m$ times. Then,
$$d(g^{m}(x),g^{m}(y))\leq Diam(w_{\sigma_{1}}\circ ... \circ w_{\sigma_{n}}\circ...\circ w_{\sigma_{1}}\circ ... \circ w_{\sigma_{n}}(X))$$ and by weak hyperbolicity we get that $d(g^{m}(x),g^{m}(y)) \longrightarrow 0$ for every $x,y \in X$ and 
this convergence is uniform in $X$. By Theorem~\ref{Jachymski}, $g$ has a unique contractive fixed point which we denote by $q_{\sigma_{1}...\sigma_{n}}$. To finish the proof, let us prove the density of the fixed points using the same arguments of Hutchinson
in \cite{H}. Consider
$A_{\sigma_{1}...\sigma_{p}}:=w_{\sigma_{1}...\sigma_{p}}(A).$
Using the invariance of $K$ one obtains that
$$K=\bigcup_{\sigma_{1}...\sigma_{p}}w_{\sigma_{1}...\sigma_{p}}(K)$$
and
$$K_{\sigma_{1}...\sigma_{p}}=\bigcup_{\sigma_{p+1}}K_{\sigma_{1}...\sigma_{p}\sigma_{p+1}}.$$
It follows that
$$K \supset K_{\sigma_{1}} \supset ... \supset K_{\sigma_{1} ... \sigma_{p}} \supset ...$$
{By compactness of $K$ and by weak hyperbolicity one obtains that this nested intersection is a singleton, which we shall denote by $k_{\sigma}$, for some $\sigma=(\sigma_1,...,\sigma_p,...)\in\Omega$.
Now, $k_{\sigma}\in K_{\sigma_{1}...\sigma_{p}}$ and $q_{\sigma_{1}...\sigma_{p}} \in K_{\sigma_{1}...\sigma_{p}}$. By weak hyperbolicity we get
$$k_{\sigma}= \lim_{p\to \infty}q_{\sigma_{1}...\sigma_{p}},$$ which ends the proof.}
\end{proof}

\section{Proof of Theorem \ref{teoremaatratorglobalmetrico}}\label{metricattractorcompacto}

We want to show that every weakly hyperbolic IFS has a measure-theoretical attractor. 
Since the iterates of the transfer operator depend on the behavior of the sequences $\Ga(\si,n,x)$, Corollary~\ref{pestrela} will be a key tool.
Indeed, in \cite{M} Mat\'e also proved that for an IFS with finite parameter space, Property $P^*$ implies the existence of a measure theoretical attractor. Here we extend his arguments for arbitrary compact 
parameter spaces.

The proof consists in showing that the iterates of a Dirac measure under the transfer operator converge to a probability measure, which is invariant by the IFS. Then, we 
proceed to show that $d(T_p^n\mu,T_p^n\delta_a)\to 0$, for any probability measure $\mu$. A key technical point is to establish the continuity of the transfer operator, which in fact is 
our first lemma.    

\begin{lemma}
\label{transfer}
If $w:\Lambda\times X\to X$ is  continuous and $X$ is compact, then for all $p\in \mathcal{P}(\Lambda)$, the transfer operator $T_p$ is continuous in the weak$^*$ topology.
\end{lemma}
\begin{proof}
Suppose that $\mu_n\to\mu$ in the weak$^*$ topology of $\mathcal{P}(X)$. We will show that $T_p\mu_n\to T_p\mu$.

Indeed, take $f\in C(X)$ and observe that 
$$\int fdT_p\mu_n=\int_{\Lambda}\int_Xf\circ w_{\lambda}d\mu_ndp=\int_X\int_{\Lambda}f\circ w_\lambda dpd\mu_n.$$
Note that the function $\Phi:X\to\R$, defined by $x\mapsto\int_{\Lambda}f\circ w_{\lambda}(x)dp$ is continuous. 

Since $\mu_n\to\mu$ in the weak$^{*}$ topology, it follows that:
$$\int_X\Phi d\mu_n\to\int_X\Phi d\mu.$$ So,
$$\int_X\int_{\Lambda}f\circ w_\lambda dpd\mu_n\to\int_X\int_{\Lambda}f\circ w_\lambda dpd\mu.$$
This completes the proof.
\end{proof}

Now, we can prove the existence of an invariant measure.

\begin{lemma}
For every $a\in X$, the sequence of measures $\{T_{p}^n(\delta_a)\}$ is convergent on the weak$^*$ topology in $\mathcal{M}(X)$. As a consequence, $\nu=\lim{\{T_{p}^n(\delta_a)\}}$ is an invariant measure for the IFS $w$.
\label{convergenciadirac}
\end{lemma}
\begin{proof}
We have to prove that $\{T_{p}^n(\delta_a)\}$ is a Cauchy sequence in $\SM(X)$. By Theorem \ref{metricacompacta} it is enough to prove that $\int_{X}fdT_{p}^n(\delta_a)$ is a Cauchy sequence of numbers, 
for every $f\in C^0(X)$ with $||f||_{0}=1$. 

By definition of the transfer operator we have that:
$$\int{fdT^n(\delta_a)}=\int_{\Lambda^n}f\circ\Gamma(\sigma,n,a)dp^n.$$
Take $n>m$. Then, using that $p$ is a probability, we get 
$$\int_{\Lambda^{n-m}}\int_{\Lambda^m}f\circ\Gamma(\sigma,m,a)dp^mdp^{n-m}=\int_{\Lambda^n}f\circ\Gamma(\sigma,m,a)dp^n.$$ Hence,
\begin{eqnarray*}
&&\left|\int{fdT^n(\delta_a)}-\int{fdT^m(\delta_a)}\right|\\
&=&\left|\int_{\Lambda^n}f\circ\Gamma(\sigma,n,a)dp^n\right.\\
\ &-&\left.\int_{\Lambda^m}f\circ\Gamma(\sigma,m,a)dp^m\right|\\
&\leq&\int_{\Lambda^n}|f\circ\Gamma(\sigma,n,a)-f\circ\Gamma(\sigma,m,a)|dp^n.
\end{eqnarray*}
Since $f$ is uniformly continuous, there exists $\delta>0$ such that if $d(x,y)<\delta$, then $|f(x)-f(y)|<\eps.$ By Corollary \ref{pestrela}, there exists $n_0=n_0(\eps)>0$ such that if $m,n\geq n_0$, then
$$d(\Gamma(\sigma,n,a),\Gamma(\sigma,m,a))<\delta.$$
Therefore, $\{T_{p}^n(\delta_a)\}$ is a Cauchy sequence. Since $\SM(X)$ is complete, there exists $\nu=\lim{\{T_{p}^n(\delta_a)\}}$. By Lemma~\ref{transfer} it follows that $\nu$ is an invariant measure.
\end{proof}

The next step is to prove that $\nu$ is in fact a measure-theoretical attractor for the IFS. 

\begin{lemma}
For all $\mu\in\mathcal{P}(X)$ and $a\in X$ the sequences $\{T^n(\delta_a)\}$ and $\{T^n(\mu)\}$ have the same limit in the weak$^*$ topology. As a consequence,  $T_p^{n}(\mu)\stackrel{n}{\longrightarrow}\nu$ in the weak$^*$ topology if $\mu \in \mathcal{M}(X).$
\label{convergenciamedidaqualquer}
\end{lemma}
\begin{proof}
As before, it is enough to show that if $||f||_{0}=1$, then $\left|\int fdT^n(\mu)-\int fdT^n(\delta_a)\right|$ goes to zero. Take $\eps>0$. Notice that
$$\int fdT^n(\mu)=\int_{\Lambda^n}\int_Xf\circ\Gamma(\sigma,n,x)d\mu dp^n.$$ Since $\mu$ is a probability we have that $$f\circ\Gamma(\sigma,n,a)=\int_Xf\circ\Gamma(\sigma,n,a)d\mu.$$
Hence, we get
\begin{eqnarray*}
\label{modulopradentrodenvo}
&&\left|\int fdT^n(\mu)-\int fdT^n(\delta_a)\right|\\
&\leq&\int_{\Lambda^n}\int_X|f\circ\Gamma(\sigma,n,x)-f\circ\Gamma(\sigma,n,x)|d\mu dp^n. 
\end{eqnarray*}
From the uniform continuity of $f$ and from Corollary \ref{pestrela} we have that the right-hand side of (\ref{modulopradentrodenvo}) is less than $\eps$ for every large $n$. This finishes the proof.
\end{proof}

Now, to conclude the proof of Theorem \ref{teoremaatratorglobalmetrico}, it only remains to prove that the support of $\nu$ is the attractor $K$.

Define for each $\lambda \in \Lambda$ the map $\eta_{\lambda}:\Om\to \Om$ by $\eta_{\lambda}(\sigma_{1},\sigma_{2},...):=(\lambda,\sigma_{1},\sigma_{2},...)$. Notice that  $\Gamma \circ \eta_{\lambda}=w_{\lambda}\circ \Gamma.$

\begin{lemma}
If $\mathbb{P}$ is the product measure in $\Om$ induced by $p\in \mathcal{M}(\Lambda)$, then we have that $\Gamma^{*}(\mathbb{P})=\nu$.
\label{medidaimagemproduto}
\end{lemma}
\begin{proof}
We will prove that $\Gamma^{*}(\mathbb{P})$ is a fixed point of the transfer operator for $w$, since we already know that $T_p$ has a unique fixed point. For that,
we begin observing that $\mathbb{P}$ is a fixed point of transfer operator to the IFS $\eta:\Lambda\times\Om\to\Om$. Since $\Gamma \circ \eta_{\lambda}=w_{\lambda}\circ \Gamma$, we can write
\begin{eqnarray*}
T_{p}(\Gamma^{*}(\mathbb{P}))&=&\int_{\Lambda}w_{\lambda}^{*}(\Gamma^{*}(\mathbb{P}))dp(\lambda)=\int_{\Lambda}\Gamma^{*}(\eta_{\lambda}^{*}\mathbb{P})dp(\lambda)\nonumber\\ 
&=&\Gamma^{*}\left(\int_{\Lambda}(\eta_{\lambda}^{*}\mathbb{P})dp(\lambda)\right)=\Gamma^{*}(\mathbb{P}).\nonumber \\
\end{eqnarray*}
This establishes the lemma.
\end{proof}
By Lemma \ref{medidaimagemproduto} and opening of the measure $p$, we get:
$$supp(\nu)=\Gamma(supp(\mathbb{P}))=K.$$
and this finishes the proof of Theorem \ref{teoremaatratorglobalmetrico}.

\section{Proof of Theorem \ref{ergodicidadedamedida}}\label{provaergodica}
To prove Theorem \ref{ergodicidadedamedida} we use the ergodicity of the shift map $\beta:\Omega\to \Omega$, which is given by 
$$\beta(\sigma_{1},\sigma_{2},...)=(\sigma_{2},\sigma_{3},...).$$ Recall that the product measure $\p$ in $\Om$ is an ergodic invariant measure for the shift map. 
The key tool for relate the shift map with the IFS is the skew product map $\tau:\Omega\times X\to \Omega\times X$, which is defined by
$$\tau(\sigma,x):=(\beta(\sigma),w_{\sigma_{1}}(x)).$$

{
Indeed, let us show how to relate ergodic averages for the IFS with ergodic averages for the skew product.
Fix $f:X\to\R$ continuous function. Let us extend $f$ to $\Omega\times X$ by $f^{\prime}:\Omega\times X\to\R$, constant in the first variable. 
In other words, $f^{\prime}(\sigma,x)=f(x)$. This implies that
\begin{equation}
\label{formulaone}
\int_{\Omega\times X}f^{\prime}d(\mathbb{P}\times\nu)=\int_Xfd\nu.
\end{equation}
Now, observe that 
\begin{eqnarray*}
f^{\prime}\left(\tau^n(\sigma,x)\right)&=&f^{\prime}\left(\beta^n(\sigma),w_{\sigma_n}\circ...\circ w_{\sigma_1}(x)\right)\\
&=&f(w_{\sigma_n}\circ...\circ w_{\sigma_1}(x)).
\end{eqnarray*}
So,
\begin{equation}
\label{relate}
 \frac{1}{n}\sum_{j=0}^{n-1}f^{\prime}\left(\tau^j(\sigma,x)\right)=\frac{1}{n}\sum_{j=0}^{n-1}f(w_{\sigma_j}\circ...\circ w_{\sigma_1}(x)).
\end{equation}
}

We have the following general result, which allows us to obtain an invariant measure for the skew product from an invariant measure for the IFS.

\begin{lemma}
\label{l.skewproduct}
Let $X$ and $\Lambda$ be compact metric spaces. 
If $\mu\in \mathcal{M}(X)$ is an invariant measure for an IFS $w:\La\times X\to X$, then the measure $\mathbb{P}\times\mu$ is invariant by $\tau$.
\end{lemma}
\begin{proof}
We want to show that for every integrable function $f:\Om\times X\to\R$ one has
\begin{equation}
\int_{\Omega\times X}f\circ\tau d(\mathbb{P}\times\mu)=\int_{\Omega\times X}f d(\mathbb{P}\times\mu).
\end{equation}
For this we shall interchange the order of integration and use a suitable split of $\Om$. To be precise, observe that the product measure in $\Om$ coincides in cylinders with the product measure in $\La\times\Om$. 
Since the $\si$-algebra of both spaces is generated by cylinders, it follows that the two measure spaces coincide. Therefore, we can split any integration in $\Om$ as an integration in $\La\times\Om$. Using this, one can write
\begin{eqnarray*}
&&\int_{\Omega\times X}f\left(\beta(\sigma),w_{\sigma_1}(x)\right)d(\mathbb{P}\times\mu)\\
&=&\int_{\Omega}\int_Xf\left(\beta(\sigma),w_{\sigma_1}(x)\right)d\mu d\mathbb{P}\\
&=&\int_{\Omega}\int_{\Lambda}\int_Xf\left(\beta(\sigma),w_{\sigma_1}(x)\right)d\mu dpd\mathbb{P}.
\end{eqnarray*}
By invariance of $\mu$ and integrability of $x\mapsto f\left(\beta(\sigma),x\right)$ for all $\sigma$, we have $$\int_{\Lambda}\int_Xf\left(\beta(\sigma),w_{\sigma_1}(x)\right)d\mu dp=\int_Xf\left(\beta(\sigma),x\right)d\mu,$$ 
for all $\sigma\in\Omega$. On the other hand, by invariance of $\mathbb{P}$ with respect to $\beta$ in $\Omega$, and integrability of $\si\mapsto f(\si,x)$ for all $x\in X$, we get
$$\int_{\Omega}f\left(\beta(\sigma),x\right)d\mathbb{P}=\int_{\Omega}f\left(\sigma,x\right)d\mathbb{P}.$$
Using these two facts and interchanging the order of integration we have that $$\int_{\Omega}\int_{\Lambda}\int_Xf\left(\beta(\sigma),w_{\sigma_1}(x)\right)d\mu dpd\mathbb{P}=\int_{\Omega}\int_{X}f\left(\sigma,x\right)d\mu d\mathbb{P}.$$ This finishes the lemma.
\end{proof}

\begin{proof}[Proof of Theorem \ref{ergodicidadedamedida}]  
Let $K\subset X$ be the unique attractor of $w$ and $\nu$ the unique invariant measure.

We want to show that for all $x\in X$, $\mathbb{P}$-q.t.p. $\sigma\in\Omega$, and for any continuous function $f:X\to\R$ we have:
\begin{equation}
\lim_{n\to\infty}\frac{1}{n}\sum_{j=0}^{n-1}f\left(w_{\sigma_j}\circ...\circ w_{\sigma_1}(x)\right)=\int_Xfd\nu. 
\label{ergodicidade}
\end{equation}
The initial step is to show that the limit on the left side of (\ref{ergodicidade}) exists for  $\mathbb{P}$-a.e. $\sigma\in\Omega$.

By Lemma \ref{l.skewproduct} and the Ergodic Theorem on $\tau$, we obtain that for $\mathbb{P}\times\nu$-a.e. $(\sigma,x)$ 
\begin{equation}
\label{amediaexiste}
f^*(\sigma,x)=\lim_{n\to\infty}\frac{1}{n}\sum_{j=0}^{n-1}f^{\prime}\left(\tau^j(\sigma,x)\right)
\end{equation}
exists. 

Consider the set $$\Omega^*=\left\{\sigma\in\Omega;\:\textrm{there exists}\:x\in X\:\textrm{such that}\:f^*(\sigma,x)\:\textrm{is defined}\right\}.$$ We claim that $\mathbb{P}(\Omega^*)=1$. 
In fact, let us suppose that for some $A\subset\Omega$, with $\mathbb{P}(A)>0$, if $\sigma\in A$, then $f^*(\sigma,x)$ do not exist for all $x\in X$. By Fubini's Theorem, this implies the existence of a set of positive $\mathbb{P}\times\nu$-measure in $\Omega\times X$ such that $f^*(\sigma,x)$ do not exist, and this is an absurd with (\ref{amediaexiste}).

Now, let us see that Corollary \ref{pestrela} implies that if $f^*(\sigma,x)$ exists for some $x\in X$ then $f^*(\sigma,y)$ also exists, for all $y\in X$, and $f^*(\sigma,x)=f^*(\sigma,y)$.

To prove this, fix $(\sigma,x)$ such that $f^*(\sigma,x)$ exists, and $y\in X$. By triangle inequality we only have to prove that
\begin{equation}
\label{estima}
\frac{1}{n}\sum_{j=0}^{n-1}\left|f^{\prime}\left(\tau^j(\sigma,x)\right)-f^{\prime}\left(\tau^j(\sigma,y)\right)\right|\to 0,
\end{equation}
when $n\to\infty$. Let us prove this. By Corollary \ref{pestrela} we have that that
for all $\delta>0$, there exists $n_0=n_0(\delta)$ such that if $n\geq n_0$, then 
$$\sup_{\alpha\in\Omega}d\left(w_{\alpha_1}\circ...\circ w_{\alpha_n}(a),w_{\alpha_1}\circ...\circ w_{\alpha_n}(b)\right)\leq\delta\:\:\:\:\textrm{for all}\:a,b\in X.$$ 
In particular, given $a,b\in X$, $\sigma\in\Omega$, and $n\geq n_0$ we have
\begin{equation}
\label{lemadoandre}
d\left(w_{\sigma_n}\circ...\circ w_{\sigma_1}(a),w_{\sigma_n}\circ...\circ w_{\sigma_1}(b)\right)\leq\delta.
\end{equation}

Now, take $\eps>0$. By uniform continuity and the above remark, there exists  $n_1>0$ such that if $n\geq n_1$, then
$$\left|f(w_{\sigma_n}\circ...\circ w_{\sigma_1}(a))-f(w_{\sigma_n}\circ...\circ w_{\sigma_1}(b))\right|<\eps,$$ for all $a,b\in X$.

Take $n_2>n_1$ such that $2\frac{n_1C}{n_2}<\eps$, where
$$C=\max_{0\leq j\leq n_1}\left\{|f(w_{\sigma_j}\circ...\circ w_{\sigma_1}(x)|,|f(w_{\sigma_j}\circ...\circ w_{\sigma_1}(y)|\right\}.$$ 
Therefore, if $n\geq n_2$, then
\begin{eqnarray*}
&&\frac{1}{n}\sum_{j=0}^{n-1}\left|f^{\prime}\left(\tau^j(\sigma,x)\right)-f^{\prime}\left(\tau^j(\sigma,y)\right)\right|\\
&=&\frac{1}{n}\sum_{j=0}^{n_1-1}\left|f^{\prime}\left(\tau^j(\sigma,x)\right)-f^{\prime}\left(\tau^j(\sigma,y)\right)\right|\\
&+&\sum_{j=n_1}^{n-1}\left|f^{\prime}\left(\tau^j(\sigma,x)\right)-f^{\prime}\left(\tau^j(\sigma,y)\right)\right|\\
&<&\frac{2n_1C}{n}+\frac{(n-n_1)\eps}{n}\\
&<&2\eps.
\end{eqnarray*}
This shows the desired.
Thus, $f^*(\sigma,x)$, for $\sigma\in\Omega^*$, is constant in $x$. Since the ergodic theorem applied to the skew product $\tau$ implies that 
$$\int_{\Omega\times X}f^*d(\mathbb{P}\times\nu)=\int_{\Omega\times X}f^{\prime}d(\mathbb{P}\times\nu),$$ by equalities (\ref{formulaone}) and (\ref{relate})
it  only remains to prove that $f^*(\sigma,x)$ is  constant for $\mathbb{P}$-a.e. $\sigma\in\Omega$. For this, we use the ergodicity of $(\be,\p)$. Indeed, if we prove that
\begin{equation}
\label{estimadois}
f^*(\be(\si),x)=f^*(\si,x),
\end{equation}
then from the ergodicity of $(\be,\p)$ it will follow that $f^*(\sigma,x)$ is constant for $\mathbb{P}$-a.e. $\sigma\in\Omega$. 

We are left to show (\ref{estimadois}). 
Let us denote by $\sum_{j=0}^{n^{-}}a_j$ the sum when $a_1$ is omitted and let $y=w_{\sigma_1}(x)$. Then we have the following estimation  
\begin{eqnarray*}
&&\left|\frac{1}{n}\sum_{j=0}^{n^{-}}f\left(w_{\sigma_j}\circ...\circ w_{\sigma_2}(x)\right)-\frac{1}{n}\sum_{j=0}^nf\left(w_{\sigma_j}\circ...\circ w_{\sigma_1}(x)\right)\right|\nonumber\\
&\leq&\frac{|f(x)|}{n}+\frac{1}{n}\sum_{j=0}^{n^{-}}\left|f\left(w_{\sigma_j}\circ...\circ w_{\sigma_2}(x)\right)-f\left(w_{\sigma_j}\circ...\circ w_{\sigma_2}(y)\right)\right|\nonumber\\
&=&\frac{|f(x)|}{n}\\
&+&\frac{1}{n}\sum_{j=0}^{n-1}\left|f\left(w_{\beta(\sigma)_j}\circ...\circ w_{\beta(\sigma)_1}(x)\right)-f\left(w_{\beta(\sigma)_j}\circ...\circ w_{\beta(\sigma)_1}(y)\right)\right|,
\end{eqnarray*}
Using the same argument applied to estimate (\ref{estima}), only using $\beta(\sigma)$ in place of $\sigma$, we see that the right side of the above inequality 
converges to zero when $n\to \infty$. This establishes (\ref{estimadois}), and completes the proof. 
\end{proof}

\section{The Complete Case}\label{completo}
In this Section we will study the more general case of complete phase space. 

\begin{definition}
Let $w:\Lambda\times X\to X$ be a continuous IFS, where $(X,d)$ is a metric space. We say that $w$ is \emph{weakly$^*$ hyperbolic}  if for all $x,y\in X$ and $\sigma\in \Om$ we have
$$\lim_{n\to +\infty}d(w_{\sigma_{1}...\sigma_{n}}(x),w_{\sigma_{1}...\sigma_{n}}(y))=0$$
and this convergence is uniform in $\Om$ and locally uniform in $X$. This means that there exists $\eta>0$ such that for all $\eps>0$ there exists $n_{0}=n_{0}(\eps)$ such that if $n\geq n_{0}$, then
$$d(w_{\sigma_{1}...\sigma_{n}}(x),w_{\sigma_{1}...\sigma_{n}}(y))<\eps,$$
for every $\sigma\in\Om$ and every $x,y$ such that $d(x,y)<\eta.$
\label{weakestrela}
\end{definition}

Our next result says that in the case of compact phase space the two definitions (weak and weak$^*$ hyperbolicity) are the same.

\begin{theorem}\label{equivdef}
Let us suppose that $\Lambda$ and $X$ are compact metric spaces. Then an IFS $w:\Lambda\times X\to X$ is weakly$^*$ hyperbolic if and only if it is weakly hyperbolic.
\end{theorem}
\begin{proof}
Suppose that $w$ is weakly hyperbolic. If $\sigma \in \Om$ and $x,y\in X$ then
$$d(w_{\sigma_{1}...\sigma_{n}}(x),w_{\sigma_{1}...\sigma_{n}}(y))\leq Diam(w_{\sigma_{1}...\sigma_{n}}(X)).$$
Since $w$ is weakly hyperbolic, we obtain that
$$\lim_{n\to \infty}d(w_{\sigma_{1}...\sigma_{n}}(x),w_{\sigma_{1}...\sigma_{n}}(y))=0.$$
By Lemma \ref{lemaprincipal} this convergence is uniform in $\Om$ and $X$, which implies that $w$ is weakly$^{*}$ hyperbolic.

Reciprocally, assume that $w$ is weakly$^{*}$ hyperbolic, and take $\sigma\in \Om$. By compactness of $X$ we get sequences $(x_{n})$ and $(y_{n})$ on $X$ such that
$$Diam(w_{\sigma_{1}...\sigma_{n}}(X))=d(w_{\sigma_{1}...\sigma_{n}}(x_{n}),w_{\sigma_{1}...\sigma_{n}}(y_{n})),\:\:\textrm{for all}\:\:n\in \N.$$
Since $\{w_{\sigma_{1}...\sigma_{n}}(X)\}$ is a nested sequence, it is enough to show that
\begin{equation}
\label{subesquence}
d(w_{\sigma_{1}...\sigma_{n_k}}(x_{n_k}),w_{\sigma_{1}...\sigma_{n_k}}(y_{n_k}))\to 0,\:\:\textrm{for some sequence}\:\:n_k\to\infty.
\end{equation} 
For this, we can use the compactness of $X$ and get subsequences $x_{n_k}\to x$ and $y_{n_k}\to y$ on $X$. We will show that $n_k$ is the desired sequence. Indeed, take $\eps>0$ and consider $\eta>0$ given by the definition of weak$^*$ hyperbolicity. There exists $k_{1}\in \N$ such that if $k\geq k_{1}$ then
\begin{eqnarray}
d(x_{n_k},x)<\eta\,\,\,\textrm{and}\,\,\,d(y_{n_k},y)<\eta.
\label{uniformlocal}
\end{eqnarray}
Since we are assuming that $w$ is a weakly$^*$ hyperbolic IFS, we obtain $k_{2}\in \N$ such that if $k\geq k_{2}$ then
\begin{eqnarray}
d(w_{\sigma_{1}...\sigma_{n_k}}(x),w_{\sigma_{1}...\sigma_{n_k}}(y))<\eps
\label{convergpontual}
\end{eqnarray}
Finally, consider $k_{0}=\max\{k_{1},k_{2}\}$. If $k\geq k_{0}$ by using (\ref{uniformlocal}), (\ref{convergpontual}) and the local uniformity of definition \ref{weakestrela} we get
\begin{eqnarray*}
d(w_{\sigma_{1}...\sigma_{n_k}}(x_{n_k}),w_{\sigma_{1}...\sigma_{n_k}}(y_{n_k}))&\leq& d(w_{\sigma_{1}...\sigma_{n_k}}(x_{n_k}),w_{\sigma_{1}...\sigma_{n_k}}(x))\nonumber \\
&+&d(w_{\sigma_{1}...\sigma_{n_k}}(x),w_{\sigma_{1}...\sigma_{n_k}}(y))\nonumber \\
&+&d(w_{\sigma_{1}...\sigma_{n_k}}(y),w_{\sigma_{1}...\sigma_{n_k}}(y_{n_k}))\nonumber \\
&\leq& 3\eps. \nonumber \\
\end{eqnarray*}
This shows that (\ref{subesquence}) holds and completes the proof.
\end{proof}

\subsection{{Proof of Theorem \ref{thm:1}}}

We shall prove the existence of an attractor in the complete case. 
Our arguments require a mild technical assumption on the phase space, which we now define.

\begin{definition}
Let $(M,d)$ be a metric space. Given $\eps>0$ and $x,y\in M$ an $\eps$-\emph{chain} joining $x$ and $y$ is a sequence $x_0=x,\:x_1,...,x_n=y$ of points in $M$ and such that $d(x_i,x_{i+1})<\eps,$ for every $i=0,...,n-1$. The number $n+1$ is the number of elements of the chain. We say that $M$ is $\eps$-\emph{chainable}, if for any $x,y\in M$ there exists an $\eps$-chain joining $x$ and $y$. 
\end{definition}

Later we shall provide examples of $\eps$-chainable metric spaces.

\begin{proof}[Proof of Theorem \ref{thm:1}]
We will prove that the Hutchinson-Barnsley operator is an asymptotic contraction on $(\mathcal{K}(X),d_H)$. Take $\eps>0$. Consider $\eta>0$ and $n_{0}=n_0(\eps)$, given by definition \ref{weakestrela}. 
Let us suppose that $d_H(A,B)<\eta$, for some  $A,B\in \mathcal{K}(X)$. We have that
$$\mathcal{F}^{n}(A)=\bigcup_{\sigma\in\Om,x\in A}w_{\sigma_{1}...\sigma_{n}}(x)\:\:\textrm{and}\:\:\mathcal{F}^{n}(B)=\bigcup_{\sigma\in\Om,y\in B}w_{\sigma_{1}...\sigma_{n}}(y).$$
If $z=w_{\sigma_{1}...\sigma_{n}}(a)$, with $a\in A$ then, using that $d_H(A,B)<\eta$, it follows that there exists $b\in B$ such that $d(a,b)<\eta$. Then, we obtain
$$d(w_{\sigma_{1}...\sigma_{n}}(a),w_{\sigma_{1}...\sigma_{n}}(b))<\eps\,\,\textrm{if}\,\,n\geq n_{0}.$$
Analogously, if $c=w_{\sigma_{1}...\sigma_{n}}(b)$, with $b\in B$, then, there exists $a\in A$ such that $d(a,b)<\eta$ and we have
$$d(w_{\sigma_{1}...\sigma_{n}}(a),w_{\sigma_{1}...\sigma_{n}}(b))<\eps\,\,\textrm{if}\,\,n\geq n_{0}.$$
Therefore,
$$d_H(\mathcal{F}^{n}(A),\mathcal{F}^{n}(B))<\eps\,\,\,\textrm{for all}\:\:n\geq n_{0}.$$ 
It remains to show that
$$d_H(\mathcal{F}^{n}(A),\mathcal{F}^{n}(B))\stackrel{n}{\longrightarrow}0$$
for any $A,B\in\SK(X)$. Since $\SK(X)$ is $\eta$-chainable, there exists compact sets $\{K_{1},...,K_{n}\}$ with $K_{1}=A,K_{n}=B$ and $d_H(K_{i},K_{i+1})<\eta\,\,\,\textrm{if}\,\,1\leq i\leq n-1$. So
\begin{eqnarray*}
d_H(\mathcal{F}^{n}(A),\mathcal{F}^{n}(B))&<&d_H(\mathcal{F}^{n}(A),\mathcal {F}^{n}(K_{2}))\\
&+&...+d_H(\mathcal{F}^{n}(K_{n-1}),\mathcal{F}^{n}(B)), 
\end{eqnarray*}
and then we have that $d_H(\mathcal{F}^{n}(A),\mathcal{F}^{n}(B))\longrightarrow 0$ when $n\to\infty$.
By Theorem~\ref{Jachymski} we have an atractor $K\in \mathcal{K}(X)$ that is also an invariant set, since $\mathcal{F}$ is continuous by Lemma \ref{continuidade}. The proof is now complete.
\end{proof}

As application we have two settings where our result applies.

\begin{corollary}
Let $(X,||.||)$ be a Banach space and $d$ its induced metric. If $w$ is a weakly hyperbolic IFS on $(X,d)$ then $\mathcal{F}$ has an attractor $K$ that is also a compact invariant set.
\label{espacosnormados}
\end{corollary}
\begin{proof}
Let us prove that $(\SK(X),d_H)$ is $\eps$-chainable for every $\eps>0$ and apply the last theorem. 
First we claim that if $B\in\mathcal{K}(X)$, and $x\in X$, then, there exists a continuous map $\psi_B:[0,1]\to \mathcal{K}(X)$ such that  $\psi_B(0)=B$ and $\psi_B(1)=\{x\}$.

To prove this claim, let us define the map $\phi:[0,1]\times X\to X$ given by $\phi(t,y)=tx+(1-t)y$ and the partial map $\phi_{t}:X\to X$ given by $\phi_{t}(x)=\phi(t,x).$ Consider the map $\psi:[0,1]\to \mathcal{K}(X)$, defined by
$$\psi(t)=\phi_{t}(B).$$
Clearly $\phi$ is continuous and so $\psi(t)$ is compact for all $t\in[0,1].$ Also, $\psi(0)=B$ and $\psi(1)=\{x\}$. It remains to prove that $\psi$ is continuous. 
In fact, given $\eps>0$ there exists $\delta>0$ such that if $|t_1-t_2|<\de$ then $d(\phi_{t_1}(b),\phi_{t_2}(b))<\eps,$ for every $b\in B$. Hence, if $|t_{1}-t_{2}|<\delta$ then  $d_H(\psi(t_{1}),\psi(t_{2}))<\eps$ which proves the continuity of $\psi$ and finishes the claim.

Given $A,B\in\SK(X)$, we can define a continuous map $\xi:[0,1]\to\SK(X)$ such that $\xi(0)=A$ and $\xi(1)=B$ as follows: fix an arbitrary point $x\in X$ and put 
$$\xi(t)=\psi_B(2t),\:\:\textrm{for}\:\:t\in[0,\frac{1}{2}]$$
and
$$\xi(t)=\psi_A(2-2t),\:\:\textrm{for}\:\:t\in[\frac{1}{2},1].$$ 
Once we have defined this continuous map, it can be easily seen that there is an $\eps$-chain joining $A$ and $B$, for every $\eps>0$. 
\end{proof}

\begin{corollary}
Let $(X,g)$ be a complete Riemannian manifold. Let $d$ be the metric induced on $X$. Suppose that $w$ is a weakly hyperbolic IFS on $(X,d)$. Then $\mathcal{F}$ has an attractor $K$ that is also a compact invariant set.
\label{variedades}
\end{corollary}
\begin{proof}
Fix a point $x\in X$. Take $B\in\SK(X)$. For any $b\in B$, consider a geodesic $\ga_b:[0,1]\to X$ joining $b$ and $x$. By a reparametrization, we can assume that the domain of every $\ga_b$ is the unity interval. Since geodesics vary smoothly, the set $\psi(t)=\{\ga_b(t);b\in B\}$ is a compact set and we have a continuous map $\psi:[0,1]\to\SK(X)$ with $\psi(0)=B$ and $\psi(1)=\{x\}$. The rest of the proof is analogous to that of the preceding corollary.   
\end{proof}

\subsection{{Proof of Theorem \ref{medidainvariantecompleto}}}

Here we give a result about invariant measures on the complete setting. Our arguments for this require a stronger form the $\eps$-chainable property, which nevertheless is satisfied by our previous examples.

\begin{definition}
Given a number ${\eps>0}$, we say that a metric space $X$ is \emph{uniformly} ${\eps}$-\emph{chainable on balls} if for every ball $B(a,r)\subset X$ there exists an integer $k=k(a,r,\eps)$ such that for every $x,y\in B(a,r)$ there exists an ${\eps}$-chain, with at most $k$ elements, joinning $x$ and $y$. 
\end{definition}

\begin{remark}
From the proofs of Corollary\ref{espacosnormados} and \ref{variedades} one sees that
every normed vector space and every complete manifold are examples of uniformly ${\eps}$-chainable metric spaces on balls, for every ${\eps}>0$. 
\end{remark}

\begin{proof}[Proof of Theorem \ref{medidainvariantecompleto}]
To prove that there exists a unique invariant measure $\nu\in\mathcal{M}(X)$ such that $supp(\nu)\subset K$, the arguments are the same used in the proof of Theorem \ref{teoremaatratorglobalmetrico} and then we only recall the main steps.
In fact, since $K$ is a compact invariant set then $w_{\lambda}(K)\subset K$ for each $\lambda\in\Lambda$ and then we can work with $w|_{K}:\Lambda\times K\longrightarrow K$. If $\mu\in\mathcal{M}(X)$ is such that $supp(\mu)\subset K$ then it is obvious from the invariance of $K$ that $supp(T_{p}(\mu))\subset K$ and then the map $T_{p}|_{K}:\mathcal{M}(K)\longrightarrow \mathcal{M}(K)$ is well defined.
\begin{enumerate}
\item The first step: For each $a\in K$, the sequence of measures $\{T_{p}^{n}(\delta_{a})\}$ is convergent on the weak$^*$ topology (or Hutchinson metric) in $\mathcal{M}(K).$
\item The second step: For each $\mu\in\mathcal{M}(K)$ and $a\in K$, the sequences $\{T_{p}^{n}(\delta_{a})\}$ and $\{T_{p}^{n}(\mu)\}$ has the same limit on the weak$^*$ topology (or Hutchinson metric).
\item The third step: The transfer operator $T_{p}$ is continuous on the weak$^*$ topology(or Hutchinson metric) in $K$.
\item The last step: If $\nu$ denotes the measure given by first step then $supp(\nu)=K.$
\end{enumerate}

It follows from the steps above that the measure $\nu$ is the only invariant measure for $w$ such that $supp(\nu)\subset K$ and in fact $supp(\nu)=K$. 

To prove the last statement we proceed as follows.

Let $\mu\in\SM(X)$ be a probability measure with compact support. We want to show that $$H(T_p^n(\mu),T_p^n(\nu))\to 0,\:\:\:\textrm{when}\:\:n\to\infty.$$ Since for any point $a\in K$, $T_p^n(\de_a)\to\nu$, when $n\to\infty$, in the Hutchinson topology, it suffices to prove that $$H(T_p^n(\mu),T_p^n(\de_a))\to 0,\:\:\:\textrm{when}\:\:n\to\infty.$$ Indeed, given $f\in\lip$, since $$\int_XfdT_p^n(\mu)=\int_X\int_{\La^n}f\circ\Ga(\si,n,x)dp^nd\mu$$ and $$\int_XfdT_p^n(\de_a)=\int_{\La^n}f\circ\Ga(\si,n,a)dp^n=\int_X\int_{\La^n}f\circ\Ga(\si,n,a)dp^nd\mu,$$ we have that 
\begin{eqnarray*}
&&\left|\int_XfdT_p^n(\mu)-\int_XfdT_p^n(\de_a)\right|\\
&\leq&\int_X\int_{\La^n}|f\circ\Ga(\si,n,x)-f\circ\Ga(\si,n,a)dp^nd\mu\\
&\leq&\int_X\int_{\La^n}d(\Ga(\si,n,x),\Ga(\si,n,a))dp^nd\mu, 
\end{eqnarray*}
and thus 
\begin{equation}
\label{cavera}
H(T_p^n(\mu),T_p^n(\de_a))\leq\int_X\xi_nd\mu,
\end{equation}
where $\xi_n(x)=\int_{\La^n}d(\Ga(\si,n,x),\Ga(\si,n,a))dp^n.$ 

Now, take $r>0$ such that $\supp(\mu)\subset B(a,r)$. Then, $\int_X\xi_nd\mu=\int_{B(a,r)}\xi_nd\mu$. We claim that $\xi_n\to 0$ uniformly in $B(a,r)$. Indeed, take $\delta>0$. Since $X$ is uniformly $\eta$-chainable on $B(a,r)$, there exists an integer $k=k(a,r,\eta)>0$ such that for every $x\in B(a,r)$ there exists an $\eta$-chain $x_0=x,...,x_n=a$, with at most $k$ elements. By weak hyperbolicity, there exists $n_0=n_0(\frac{\delta}{k})>0$ such that $n\geq n_0$ implies that $$d(\Ga(\si,n,x),\Ga(\si,n,y))\leq\frac{\delta}{k},$$ for every $\si\in\Om$ and for every pair $x,y\in X$ with $d(x,y)<\eta$. Therefore, if $n\geq n_0$ we have that 
\begin{eqnarray*}
&&d(\Ga(\si,n,x),\Ga(\si,n,a))\\
&\leq&\sum_{j=0}^nd(\Ga(\si,n,x_j),\Ga(\si,n,x_{j+1}))\\
&\leq&\sum_{j=0}^n\frac{\delta}{k}<\delta,
\end{eqnarray*} for every $\si\in\Om$, and it follows that $$\int_{\La^n}d(\Ga(\si,n,x),\Ga(\si,n,a))dp^n<\delta,$$  for every $x\in B(a,r)$. This proves our claim.

By claim and inequality (\ref{cavera}) we conclude that 
$$H(T_p^n(\mu),T_p^n(\de_a))\to 0,$$ finishing the proof. 
\end{proof}

\section{Drawing the Attractor}\label{s.caos}

Here we take inspiration from \cite{BV1} to give a result about how to visualize the attractor through orbits of the IFS instead of computing the full Hutchinson-Barnsley operator.
{The result we shall prove is closely related with the so-called chaos game, which is studied in \cite{BV1} for IFS with finite parameter space.}

For the convenience of the reader, let us recall some definitions given in the introduction.

\begin{definition}
An \emph{orbit of the IFS} starting at {a point $x\in X$} is a sequence {$\{x_k\}_{k=0}^{\infty}\subset X$} such that $x_0=x$, $x_{k+1}=w_{\lambda_k}(x_k)$, 
for some sequence $\{\lambda_k\}_{k=1}^{\infty}\in\Om$ in the parameter space.
\end{definition}

\begin{definition}
Given an IFS $w:\Lambda\times X\to X$ with attractor $A$, we say that an orbit starting at $x$ \emph{draws} the attractor if $$A=\lim_{k\to\infty}\{x_n\}_{n=k}^{\infty},\:\:\:\textrm{in the Hausdorff metric}.$$
\end{definition}
 
Given an IFS $w:\Lambda\times X\to X$ with attractor $A$ (with basin $U$) and a point $x\in X$ we shall denote by $\cA(x)\subset\Om$ the set formed by the sequences 
$\{\lambda_k\}_{k=1}^{\infty}$ such that the correspondent orbit $x_0=x$, $x_k=w_{\lambda_k}(x_{k-1})$ draws the attractor. 

{Our goal is to prove that there exists a ``large'' set of orbits which draws the attractor of the IFS. More precisely, consider a probability measure $m$
in the space $\Om$. We would like to say that for every $x\in X$, the set $\cA(x)$ has full measure. 
In \cite{BV1} the authors consider, for a parameter space $\Lambda=\{1,...,N\}$, a probability measure $m$ over $\Om$ such that there exists $a\in(0,1/N]$ such that the set
$$\{\sigma\in\Om;\forall\:n,j\:\textrm{there exists a probability at least}\:a\:\textrm{that}\:\sigma_n=j\}$$
has total probability. 
For instance, the Bernoulli measures in $\Om$ are an example of such measures. In our case, we consider measures in the parameter space with a uniform lower bound for the size of balls.}

\begin{definition}
We say that a probability $p\in\cP(\La)$ is \emph{fair} if there exists a positive function $f:(0,+\infty)\to(0,1]$ such that $$p\left(B(\lambda,\delta)\right)\geq f(\delta),\:\:\:\textrm{for every}\:\:\:\lambda\in\Lambda.$$ 
\end{definition}

{There are plenty of examples of fair measures, such as} the Lebesgue measure in $\R^n$ and the Haar measure of a Lie group. We will use $\p=p^{\N}$, as before.

Recall that a metric space is said to be \emph{proper} if every closed ball is compact. With these notations, Theorem~\ref{caosgamecor} restates as follows.

\begin{theorem}
\label{caosgame}
Let $X$ be a proper complete metric space, and $w:\Lambda\times X\to X$ a continuous IFS. Suppose that $w$ has an attractor $A$ with local basin of attraction $U$. Then, 
for every point $x\in U$, $\p\left(\cA(x)\right)=1$.  
\end{theorem}

We remark that the class of probabilities used in \cite{BV1} seems to be more general than {those we consider here}, 
but we don't have any definitive assertion about this. Before proving Theorem \ref{caosgame} we prepare some lemmas. The first one is quite elementary and we left its proof to the reader.

\begin{lemma}
Let $X$ be a complete metric space and $C\subset X$ compact. Then $X$ is proper if and only if $B(C,r)$ is compact for every $r>0$
\end{lemma}

From now on, we assume that we are under the assumptions of Theorem~\ref{caosgame}. Note that, since $\cF$ is continuous (see Lemma~\ref{continuidade}), we have that $A=\cF(A)$. The next lemma provides some (uniform) 
control for the speed of convergence of the iterates $\cF^k(\{x\})$ to the attractor, but for points $x$ close to the attractor. This control will be one of the key points to prove Theorem~\ref{caosgame}. This lemma was proved in \cite{BV1} for 
$\La:=\{1,...,N\}$. The proof is the same, and we give it here just for the sake of completeness.

\begin{lemma}
\label{tempodechegadaaoatrator}
Let $w:\Lambda\times X\to X$ be a continuous IFS of a proper complete metric space $X$ and compact parameter space $\Lambda$. Suppose that $w$ has a local attractor $A$ with local basin $U$. Then for any $\eps>0$ there exists an integer $N=N(\eps)$ 
such that for any $x\in(B(A,\eps))\cap U$ there is an integer $m=m(x,\eps)<N$ such that 
$$d_H\left(A,\cF^m(\{x\})\right)<\frac{\eps}{4}.$$  
\end{lemma}
\begin{proof}
Without loss of generality we assume that $B(A,\eps)\subset U$. If $x\in B(A,\eps)$ then, there exists an integer $m=m(x,\eps)\geq 0$ such that $$d_H\left(A,\cF^m(\{x\})\right)<\frac{\eps}{8},$$ by definition of an atractor. Since $\cF$ is continuous, there exists $r_{x}>0$ such that for every $y\in B(x,r_{x})$ we have $$d_H\left(\cF^m(\{x\},\cF^m(\{y\})\right)<\frac{\eps}{8},$$ and thus $d_H\left(A,\cF^m(\{y\})\right)<\frac{\eps}{4}$, for every $y\in B(x,r_x)$. Since $X$ is proper, $B(A,\eps)$ is compact and so there is a finite set $\{x_1,...,x_n\}$ such that $$B(A,\eps)\subset\bigcup_{i=1}^{n}B(x_i,r_{x_i}).$$ Let $N=\max\{m(x_i,\eps);i=1,...,n\}+1$. Then, for every $x\in B(A,\eps)$, there is $i\in\{1,...,n\}$ such that $x\in B(x_i,r_{x_i})$ and therefore $$d_H\left(A,\cF^m(\{x\})\right)<\frac{\eps}{4},$$ with $m=m(x_i,\eps)<N$. This proves the lemma.
\end{proof}

Now, we will use continuity of the IFS $w$ to control orbits of nearby points. The main issue here is that this can be done uniformly in $B(A,\eps)$ due to compactness.

\begin{lemma}
\label{czeroboladotentativadois}
Let $w:\Lambda\times X\to X$ be a continuous IFS of a proper complete metric space $X$ and compact parameter space $\Lambda$. Suppose that $w$ has an attractor $A$ with local basin $U$. For every $\eps>0$, and every integer $N>0$, there exists $\delta=\delta(\eps,N)$ such that for every $m<N$, if $x\in B(A,\eps)$ and $d(\sigma_i,\lambda_i)<\delta$ in $\Lambda$, $i=1,...,m$ then $$d\left(w_{\lambda_m}\circ...\circ w_{\lambda_1}(x),w_{\sigma_m}(x)\circ...\circ w_{\sigma_1}(x)\right)<\frac{\varepsilon}{4}.$$
\end{lemma}
\begin{proof}
Fix $\eps>0$. The proof goes by induction on $N$. Since $B(A,\eps)$ is compact, the case $N=1$ follows by uniform continuity. Suppose that the the lemma holds for $N$, and let us prove that it also holds for $N+1$. Again, since 
$Y=\cup_{n=0}^N\cF^n(B(A,\eps))$ is a compact metric space, $w$ restricted to this set is uniformly continuous. Hence, there exists $\delta_1=\delta_1(\eps,N)>0$ such that if $\la_N,\si_N\in\La$ and $a,b\in Y$ with $d(\lambda_N,\sigma_N)<\delta_1$ in
$\Lambda$ and $d\left(a,b\right)<\delta_1$ in $Y$, then
$$d(w_{\sigma_N}(a),w_{\lambda_N}(b))<\eps.$$ By the induction hypothesis, there exists $\delta_2=\delta_2(N,\eps)$ such that if $d(\lambda_i,\sigma_i)<\delta_2$ for every $i=1,...,N-1$, then 
$$d\left(w_{\lambda_{N-1}}\circ...\circ w_{\la_1}(x),w_{\si_{N-1}}\circ...\circ w_{\si_1}(x)\right)<\delta_1.$$ Therefore, if $\delta=\min\{\delta_1,\delta_2\}$ and $d(\lambda_i,\sigma_i)<\delta$ for every $i=1,...,N$,
it follows that 
$$d\left(w_{\lambda_{N}}\circ...\circ w_{\la_1}(x),w_{\si_{N}}\circ...\circ w_{\si_1}(x)\right)<\eps,$$ and thus the case $N+1$ is true. This completes the proof.
\end{proof}

The next lemma reduces the proof of Theorem~\ref{caosgame} to a proof of a simpler statement.

\begin{lemma}
 \label{l.reducao}
Fix a point $x\in X$. Suppose that the following property holds
\begin{itemize}
 \item for each $\eps>0$ there exists $K_{\eps}>0$ such that for every $L\geq K_{\eps}$ one has a set 
$B_L\subset\Om$ with $\p(B_L)=1$, and such that if $\si\in B_L$ and if $x_k=w_{\sigma_k}(x_{k-1})$ is the $\sigma$-orbit of $x$ then $d_H(A,\{x_k\}_{k\geq L})<\eps$.
\end{itemize}
Then, there exists $\cB\subset\Om\cap\cA(x)$ with $\p(\cB)=1$.
\end{lemma}
\begin{proof}
Let $\cB_{\eps}=\cap_{L\geq K_{\eps}}B_L$. Clearly we have  $\cB_{\eps}\subset\Om$, 
with $\p(\cB_{\eps})=1$. Moreover, for every $x$-orbit $\{x_{k+1}=w_{\si_k}(x_k)\}$, generated by some sequence $\si=(\si_k)\in\cB_{\eps}$ satisfies $d_H(A,\{x_k\}_{k\geq L})<\eps$, 
for every $L\geq K_{\eps}$. Now, we take $\eps_n=\frac{1}{n}$ and define $\cB=\cap_{n}\cB_{\eps_n}$. Obviously, $\p(\cB)=1$. Moreover, it is easy to see that $\cB\subset\cA(x)$. Indeed, 
take $\si\in\cB$ and 
consider $\{x_k\}$ the orbit of $x$ generated by $\si$. For any $\eps>0$ we can take a large $n$ with $\eps_n<\eps$. Since $\si\in\cB_{\eps_n}$, we have that $L\geq K_{\eps_n}$ implies 
$d_H(A,\{x_k\}_{k\geq L})<\eps_n<\eps$. Thus, $A=\lim_{L\to\infty}\{x_k\}_{k\geq L}$, wich proves that $\cB\subset\cA(x)$. This establishes the lemma.
\end{proof}

Now, we give the proof of Theorem \ref{caosgame}.

\begin{proof}[Proof of Theorem \ref{caosgame}]
{We shall apply Lemma~\ref{l.reducao}.} So, let us fix $\eps>0$ and exhibit the integer $K_{\eps}$. By definition of an attractor, there exists $K_{\eps}$ such that $k\geq K_{\eps}$ implies that $$d_H(\cF^k(\{x\}),A)<\eps,$$ in particular, given any sequence $\{\la_k\}_{k=1}^{\infty}\in\Om$, the correspondent orbit satisfies $$x_k\in\cF^k(\{x\})\subset B(A,\eps),$$ for every $k\geq K_{\eps}$. Take $L\geq K_{\eps}$ and let us construct the set $B_L$. 

The key observation is that for any point $a$ in $B(A,\eps)$, we can find a finite sequence of parameters that ``corrects" the orbit of $a$, making it visit every portion of $A$. 

To be precise, consider a set $\{a_1,...,a_l\}\subset A$ such that $A\subset\cup_{j=1}^lB(a_j,\frac{\eps}{4})$. Observe that if a set $R\subset B(A,\eps)$ has non-empty intersection with every ball $B(a_j,\frac{\eps}{2})$ then $d_H(A,R)<\eps$. In virtue of this, we say that a finite word $\{\si_1,...,\si_n\}\subset\La$ \emph{corrects} a point $a$ if there exists $n_1,\dots,n_l\subset\{1,...,n\}$ such that $$w_{\si_{n_j}}\circ\dots\circ w_{\si_1}(a)\in B(a_j,\frac{\eps}{2}).$$
Now, observe that Lemma \ref{tempodechegadaaoatrator} implies that for each $a\in B(A,\eps)$ there is a finite word $\la_1,...,\la_m$, such that $$w_{\la_m}\circ\dots\circ w_{\la_1}(a)\in B(a_1,\frac{\eps}{4}),$$ and the length $m$ of this word is bounded by some constant $N=N(\eps)$. Applying the same reasoning with $w_{\la_m}\circ\dots\circ w_{\la_1}(a)$ in the place of $a$, we find a second finite word, with the same bound on its length, so that the orbit of $a$ under this two blocks of words now visits both balls $B(a_1,\frac{\eps}{4})$ and $B(a_2,\frac{\eps}{4})$. Continuing in this way we can find a finite correcting word with length at most $lN$, which means that the orbit of $a$ under this word visits every ball $B(a_j,\frac{\eps}{4})$. 

Also, by Lemma \ref{czeroboladotentativadois} there exists $\de=\de(\eps,N)$ such that for every finite word with the same length of this correcting word and $\de$-close to it, the correspondent orbit of $a$ visits every ball $B(a_j,\frac{\eps}{2})$.  

Since $p$ is a fair measure, we have that the $\p$-measure of the set $$C_0=\left\{\si\in\Om;\:\si_{L+1},...,\si_{L+lN}\:\textrm{corrects}\:x_L\right\}$$ is at least $f(\de)^{lN}$. By the same reason, the measure of each set $$C_j=\left\{\si\in\Om;\:\si_{L+jlN+1},...,\si_{L+(j+1)lN}\:\textrm{corrects}\:x_{L+jlN}\right\}$$ is at least $f(\de)^{lN}$. Moreover, since this sets are independent events, it follows that  $$p(\bigcap_{j=0}^{\infty}(\Om-C_j))\leq p(\bigcap_{j=0}^{t-1}(\Om-C_j))\leq(1-f(\de)^{lN})^t.$$ Therefore, $p(\cup_jC_j)=1$. 

By construction of the sets $C_j$, for every $\si\in\cup_jC_j$ the correspondent orbit satisfies $$A\subset B\left(\{x_k\}_{k\geq L},\eps\right),$$ and since $L\geq K_{\eps}$ we also have that $\{x_k\}_{k\geq L}\subset B(A,\eps)$. Thus $$d_H(A,\{x_k\}_{k\geq L})<\eps.$$ Therefore, putting $B_L=\cup_jC_j$ the result is proved.  
\end{proof}

\paragraph{Acknowledgments.}
We would like to thank Katrin Gelfert and Daniel Oliveira for presenting us the paper of Kravchenko \cite{K}.

\bigskip
\noindent
\begin{minipage}[t]{12cm}
{\bf Alexander Arbieto}\\ 
Universidade Federal do Rio de Janeiro\\
Brazil\\[2mm]
E-mail:arbieto@im.ufrj.br
 
\bigskip
\noindent
{\bf Andr\'e Junqueira}\\ 
Universidade Federal de Vi\c{c}osa\\
Brazil\\[2mm]
E-mail: andre.junqueira@ufv.br

\bigskip
\noindent
{\bf Bruno Santiago}\\ 
Universit\'e de Bourgogne\\
France\\[2mm]
E-mail: bruno.rodrigues-santiago@u-bourgogne.fr

\end{minipage}
\end{document}